\documentclass[11pt,a4paper]{article}

\usepackage{inputenc}
\usepackage{amsmath}
\usepackage{bm}
\usepackage{bbold}
\usepackage{amsthm}

\usepackage{hyperref}

\title{Products of Random Matrices and Queueing System Performance Evaluation\thanks{Proc. 4th St.~Petersburg Workshop on Simulation / Ed. by S.~M.~Ermakov, Yu.~N.~Kashtanov, V.~B.~Melas, NII Chemistry St.~Petersburg University Publishers, St.~Petersburg, 2001, pp.~304--309.}} 

\author{N.~K.~Krivulin\thanks{Faculty of Mathematics and Mechanics, St.~Petersburg State University, 28 Universitetsky Ave., St.~Petersburg, 198504, Russia, 
nkk@math.spbu.ru.} \thanks{The work was partially supported by the Russian Foundation for Basic Research, Grant~\#00-01-00760.}
}

\date{}

\newtheorem{theorem}{Theorem}
\newtheorem{lemma}[theorem]{Lemma}

\begin{document}

\maketitle

\begin{abstract}
We consider (max,+)-algebra products of random matrices, which arise from 
performance evaluation of acyclic fork-join queueing networks. A new algebraic
technique to examine properties of the product and investigate its limiting
behaviour is proposed based on an extension of the standard matrix 
(max,+)-algebra by endowing it with the ordinary matrix addition as an 
external operation. As an application, we derive bounds on the (max,+)-algebra 
maximal Lyapunov exponent which can be considered as the cycle time of the 
networks.
\\

\textit{Key-Words:} (max,+)-algebra, product of random matrices, maximal Lyapunov 
exponent, acyclic fork-join networks, cycle time
\end{abstract}

\section{Introduction}
We consider $(\max,+)$-algebra products of random matrices arising from 
performance evaluation of acyclic fork-join queueing networks. The problem is 
to examine limiting behaviour of the product so as to evaluate its limiting 
matrix and the maximal Lyapunov exponent normally referred to as the system 
cycle time.

In order to investigate the products, we develop a pure algebraic technique 
similar to those involved in the conventional linear algebra. The technique is 
based on an extension of the standard matrix $(\max,+)$-algebra 
\cite{Vorobjev1967Extremal,Cuninghame-Green1979Minimax,Baccelli1993Synchronization,Maslov1994Idempotent} by endowing it with the ordinary matrix 
addition as an external operation. New properties of the extended algebra are 
then established in the form of inequalities, which may find their 
applications beyond of the scope of the current topic. We conclude the paper 
with an example of application of the proposed technique to establish bounds 
on the cycle time and on its related limiting matrix in fork-join queueing 
networks.

In fact, there exist similar results on evaluation of the Lyapunov exponent 
(see, e.g., \cite{Baccelli1993Synchronization} and references therein). However, they are 
essentially based on the description of system dynamics and related proofs 
made in terms of either Petri nets or stochastic events graphs. On the 
contrary, we exploit a different approach (see \cite{Krivulin2000Algebraic} for farther 
details) based on pure algebraic techniques. It allows one to write and handle 
the dynamic equations directly without having recourse to an intermediate 
description in the Petri nets or in another tedious language. 

\section{Motivating Example and Algebraic Model}\label{S-AFJQ}
Consider a network of $ n $ nodes, with its topology described by an 
oriented acyclic graph. The nodes that have no predecessors are assumed to 
represent an infinite external arrival stream of customers. Each node without 
successors is considered as an output node which releases customers from the 
network.

Each node has a server and infinite buffer operating as a single-server queue
under the first-come, first-served discipline. At the initial time, the
servers and their buffers are assumed to be free of customers, except for the
buffers in nodes with no predecessors, each assumed to have an infinite number 
of customers.

The operation of each node can include join and fork operations which are 
performed respectively before and after service. The join operation is 
actually thought to cause each customer which comes into a node not to enter 
the queue but to wait until at least one customer from all preceding nodes 
arrives. Upon arrival, these customers are replaced by a new customer which
joins the queue.

The fork operation at a node is initiated every time the service of a customer 
is completed. It consists in replacing the customer by several new customers, 
each intended to go to one of the subsequent nodes. 

For the queue at node $ i $, we denote the $k$th departure epochs by 
$ x_{i}(k) $, and the $k$th service time by $ \tau_{ik} $. We assume 
$ \tau_{ik} $ to be a given nonnegative random variable (r.v.) for all 
$ i=1,\ldots,n $, and $ k=1,2,\ldots $

We are interested in evaluating the limit
$$
\gamma=\lim_{k\to\infty}\frac{1}{k}\max_{i}x_{i}(k),
$$
which is normally referred to as the cycle time of the network. 

In order to represent the network dynamics in a form suitable for further
analysis, we exploit the idempotent $(\max,+)$-algebra based approach 
developed in \cite{Krivulin2000Algebraic}.

The $(\max,+)$-algebra \cite{Vorobjev1967Extremal,Cuninghame-Green1979Minimax,Baccelli1993Synchronization} presents a triple
$ \langle R_{\varepsilon},\oplus,\otimes\rangle $ with
$ R_{\varepsilon}=R \cup \{\varepsilon\} $,
$ \varepsilon=-\infty $, and operations $ \oplus $ and $ \otimes $
defined for all $ x,y\in R_{\varepsilon} $ as
$$
x \oplus y=\max(x,y), \quad x \otimes y=x + y.
$$

The $(\max,+)$-algebra of matrices is introduced in the ordinary way. The 
square matrix $ {\cal E} $ with all its elements equal $ \varepsilon $ 
presents the null matrix, whereas the matrix $ E={\rm diag}(0,\ldots,0) $ 
with $ \varepsilon $ as its off-diagonal elements is the identity.

Let us denote the vector of the $k$th customer departures from the network 
nodes by 
$ \bm{x}(k)=(x_{1}(k),\ldots,x_{n}(k))^{T} $, and introduce 
the matrix $ {\cal T}_{k}={\rm diag}(\tau_{1k},\ldots,\tau_{nk}) $ with 
all its off-diagonal elements equal $ \varepsilon $. 
 
As it has been shown in \cite{Krivulin2000Algebraic}, the dynamics of acyclic 
fork-join networks can be described by the stochastic difference equation
\begin{equation}\label{E-SDEQ}
\bm{x}(k)=A(k)\otimes\bm{x}(k-1), 
\qquad
A(k)
=
\bigoplus_{j=0}^{p}({\cal T}_{k}\otimes G^{T})^{j}\otimes{\cal T}_{k},
\end{equation}
where $ G $ is a matrix with the elements
$$
g_{ij}=\left\{\begin{array}{ll}
        0, & \mbox{if there exists arc $ (i,j) $ in the network graph}, \\
        \varepsilon, & \mbox{otherwise},
       \end{array}\right.
$$
and $ p $ is the length of the longest path in the graph. 

The matrix $ G $ is normally referred to as the support matrix of the
network. Note that since the network graph is acyclic, we have 
$ G^{q}={\cal E} $ for all $ q>p $.

Consider the service cycle time $ \gamma $. Now we can represent it as
$$
\gamma=\lim_{k\to\infty}\frac{1}{k}\|\bm{x}(k)\|,
$$
where $ \|\bm{x}(k)\|=\max_{i}x_{i}(k) $.

In order to get information about the growth rate of 
$ \bm{x}(k) $, we will examine the limiting behaviour of the 
matrix 
$$
A_{k}
=
A(k)\otimes\cdots\otimes A(1)
=
\bigotimes_{i=1}^{k}\bigoplus_{j=0}^{p}
({\cal T}_{k}\otimes G^{T})^{j}\otimes{\cal T}_{k}.
$$

\section{Distributivity Properties and Matrix Products}\label{S-DIPR}

Let $ A_{ij} $ be $(n\times n)$-matrices for all $ i=1,\ldots,k $ and
$ j=1,\ldots,m $. Distributivity of the operation $ \otimes $ over
$ \oplus $ immediately gives the equality
\begin{equation}\label{E-DMOA}
\bigotimes_{i=1}^{k}\bigoplus_{j=1}^{m} A_{ij}
=
\bigoplus_{1\leq j_{1},\ldots,j_{k}\leq m}
A_{1j_{1}}\otimes\cdots\otimes A_{kj_{k}},
\end{equation}
which leads, in particular, to the inequality
\begin{equation}\label{I-DMOA}
\bigotimes_{i=1}^{k}\bigoplus_{j=1}^{m} A_{ij}
\geq
\bigoplus_{j=1}^{m}\bigotimes_{i=1}^{k} A_{ij}.
\end{equation}

We consider the ordinary matrix addition $ + $ as an external operation,
and assume $ \otimes $ and $ \oplus $ to take precedence over $ + $. 
In a similar way as above, we have
\begin{eqnarray}
\sum_{i=1}^{k}\bigoplus_{j=1}^{m} A_{ij}
&=&
\bigoplus_{1\leq j_{1},\ldots,j_{k}\leq m}(A_{1j_{1}}+\cdots+A_{kj_{k}}),
\label{E-DAOM} \\
\sum_{i=1}^{k}\bigoplus_{j=1}^{m} A_{ij}
&\geq&
\bigoplus_{j=1}^{m}\sum_{i=1}^{k} A_{ij}.\label{I-DAOM}
\end{eqnarray}

Let $ G_{1} $ and $ G_{2} $ be support matrices. For any matrices 
$ A $ and $ B $, we have
\begin{equation}\label{E-GABG}
G_{1}\otimes(A+B)\otimes G_{2} 
\leq
G_{1}\otimes A\otimes G_{2}+G_{1}\otimes B\otimes G_{2}. 
\end{equation}

Assume $ D_{1} $ and $ D_{2} $ to be diagonal matrices with all 
off-diagonal elements equal $ \varepsilon $. Then for any matrices $ A $ 
and $ B $, it holds
\begin{equation}\label{E-DPOA}
D_{1}\otimes(A+B)\otimes D_{2}
=
D_{1}\otimes A\otimes D_{2}+B
=
D_{1}\otimes A+B\otimes D_{2}
=
A+D_{1}\otimes B\otimes D_{2}.
\end{equation}

Now we examine products of alternating diagonal and support matrices denoted
respectively by $ D $ and $ G $, which take the form
$$
D\otimes\underbrace{(G\otimes D)\otimes\cdots
\otimes(G\otimes D)}_{\mbox{\scriptsize $m$ times}}
=
D\otimes(G\otimes D)^{m}.
$$

In order to simplify further formulas, we introduce the following notations
$$
\Phi_{j}(D)=D\otimes(G\otimes D)^{j}, 
\qquad
\Psi_{i}^{j}(D)=G^{i}\otimes D\otimes G^{j}.
$$

First assume the diagonal matrices to have both positive and negative entries
on the diagonal. The next lemma can be proved using (\ref{E-GABG}) and
induction on $ m $.

\begin{lemma}\label{L-DP2S}
It holds that
$$
\Phi_{m}(D)
\leq
\sum_{j=0}^{m}\Psi_{j}^{m-j}(D).
$$
\end{lemma}

Furthermore, assuming $ D_{i} $, $ i=1,\ldots,k $, to be diagonal
matrices, one can obtain the next result based on Lemma~\ref{L-DP2S} and 
inequality (\ref{E-GABG}).
\begin{lemma}\label{L-DO2S}
Let $ m_{1},\ldots,m_{k} $ be integers, and $ m=m_{1}+\cdots+m_{k} $. 
Then it holds
$$
\bigotimes_{i=1}^{k}\Phi_{m_{i}}(D_{i})
\leq
\sum_{i=1}^{k}\sum_{j=M_{i-1}}^{M_{i}}\Psi_{j}^{m-j}(D_{i})
$$
with $ M_{0}=0 $, $ M_{i}=m_{1}+\cdots+m_{i} $, $ i=1,\ldots, k $.
\end{lemma}

Let the matrices $ D_{1},\ldots,D_{k} $ have only nonnegative elements on
the diagonal. With (\ref{E-DPOA}) and (\ref{E-GABG}), one can prove the next
lemma.
\begin{lemma}\label{L-DPLD}
Suppose that $ m_{1}+\cdots+m_{r}=m_{r+1}+\cdots+m_{k}=m $ with
$ m-m_{r}\leq m_{r+1} $ for some $ r $. Then for any integer $ s $
such that $ m-m_{r}\leq s\leq m_{r+1} $, it holds
$$
\bigotimes_{i=1}^{r}\Phi_{m_{i}}(D_{i})
+
\bigotimes_{i=r+1}^{k}\Phi_{m_{i}}(D_{i})
\geq
\bigotimes_{i=1}^{k}\Phi_{s_{i}}(D_{i})
$$
with $ s_{1}+\cdots+s_{k}=m $, and
$$
s_{i}
=\left\{
   \begin{array}{ll}
    m_{i},     & \mbox{if $ 1\leq i<r$}, \\
    s-m+m_{r}, & \mbox{if $ i=r$}, \\
    m_{r+1}-s, & \mbox{if $ i=r+1$}, \\
    m_{i},     & \mbox{if $ r+1<i\leq k$}.
   \end{array}
 \right.
$$
\end{lemma}

\section{Subadditivity Property and Algebraic Bounds}\label{S-SPAB}

Consider the family
$ \{A_{lk}^{T}| l,k=0,1,\ldots; l<k\} $ of matrices
$$
A_{lk}^{T}
=
A^{T}(l+1)\otimes\cdots\otimes A^{T}(k), \quad A_{0k}^{T}=A_{k}^{T}.
$$

The next lemma states that the family $ \{A_{lk}^{T}\} $ possesses
subadditivity property.
\begin{lemma}\label{L-SUBA}
For all $ l<r<k $, it holds
$$
A_{lk}^{T}
\leq
A_{lr}^{T}+A_{rk}^{T}.
$$
\end{lemma}

\begin{proof}
By applying (\ref{E-DMOA}) and (\ref{E-DAOM}), and then Lemma~\ref{L-DPLD}, we 
have
\begin{multline*}
A_{lr}^{T}+A_{rk}^{T}
=
\bigotimes_{i=l+1}^{r}\bigoplus_{j=0}^{p}\Phi_{j}({\cal T}_{i})
+
\bigotimes_{i=r+1}^{k}\bigoplus_{j=0}^{p}\Phi_{j}({\cal T}_{i})
\\
\geq
\bigoplus_{m=0}^{p}\bigoplus_{s_{l+1}+\cdots+s_{k}=m}
\bigotimes_{i=l+1}^{k}\Phi_{s_{i}}({\cal T}_{i}).
\end{multline*}

Finally, since $ G^{m}={\cal E} $ for all $ m>p $, we get
$$
A_{lr}^{T}+A_{rk}^{T}
\geq
\bigoplus_{0\leq s_{l+1},\ldots,s_{k}\leq p}
\bigotimes_{i=l+1}^{k}\Phi_{s_{i}}({\cal T}_{i})
=
\bigotimes_{i=l+1}^{k}\bigoplus_{j=0}^{p}\Phi_{j}({\cal T}_{i})
=
A_{lk}^{T}.
\qedhere
$$
\end{proof}

The next lemma offers bounds on $ A_{k}^{T} $.
\begin{lemma}\label{L-LUBA}
It holds that
$$
\bigoplus_{r=0}^{\lfloor p/k\rfloor}
\bigotimes_{i=1}^{k}\Phi_{r}({\cal T}_{i})
\leq
A_{k}^{T}
\leq
\Big\|\bigoplus_{i=1}^{k}{\cal T}_{i}\Big\|\otimes
\bigoplus_{r=1}^{p}G^{r} 
+
\sum_{i=1}^{k}\bigoplus_{0\leq r+s\leq p}\Psi_{r}^{s}({\cal T}_{i}),
$$
where $ \lfloor r\rfloor $ denotes the greatest integer equal to or less
than $ r $.
\end{lemma}

\begin{proof}
The lower bound is an immediate consequence from (\ref{I-DMOA}), and the
condition that $ G^{m}={\cal E} $ if $ m=kr>p $.

In order to derive the upper bound, we first apply (\ref{E-DMOA}) to write
$$
A_{k}^{T}
=
\bigoplus_{0\leq m_{1},\ldots,m_{k}\leq p}
\bigotimes_{i=1}^{k}\Phi_{m_{i}}({\cal T}_{i}).
$$

Application of Lemma~\ref{L-DO2S} gives
\begin{multline*}
\bigotimes_{i=1}^{k}\Phi_{m_{i}}({\cal T}_{i})
\leq
\sum_{i=1}^{k}\sum_{j=M_{i-1}}^{M_{i}}\Psi_{j}^{m-j}({\cal T}_{i})
\\
=
\sum_{i=1}^{k}\sum_{j=M_{i-1}+1}^{M_{i}}\!\Psi_{j}^{m-j}({\cal T}_{i})
+
\sum_{i=1}^{k}\Psi_{M_{i-1}}^{m-M_{i-1}}({\cal T}_{i}).
\end{multline*}

With (\ref{I-DAOM}), we further obtain
$$
A_{k}^{T}
\leq
\bigoplus_{0\leq m_{1},\ldots,m_{k}\leq p}
\sum_{i=1}^{k}\sum_{j=M_{i-1}+1}^{M_{i}}\Psi_{j}^{m-j}({\cal T}_{i})
+
\bigoplus_{0\leq m_{1},\ldots,m_{k}\leq p}
\sum_{i=1}^{k}\Psi_{M_{i-1}}^{m-M_{i-1}}({\cal T}_{i}).
$$

It remains to replace the first sum with its obvious upper bound, and then 
apply (\ref{E-DAOM}) and (\ref{I-DAOM}) to the second sum so as to get the
desired result.
\end{proof}

\section{Evaluation of Bounds on the Cycle Time}\label{S-EBCT}
The next statement follows from the classical result in \cite{Kingman1973Subadditive}, 
combined with Lemma~\ref{L-SUBA}.
\begin{theorem}\label{T-EXIS}
If $ \tau_{i1},\tau_{i2},\ldots $, are i.i.d. r.v.'s with
$ \mbox{\rm E}[\tau_{i1}]<\infty $ for each $ i=1,\ldots,n $, then there 
exists a fixed matrix $ A $ such that with probability 1,
$$
\lim_{k\to\infty}A_{k}^{T}/k=A^{T}, 
\quad \mbox{and} \quad
\lim_{k\to\infty}\mbox{\rm E}[A_{k}^{T}]/k=A^{T}.
$$
\end{theorem}

Furthermore, application of Lemma~\ref{L-LUBA} together with asymptotic 
results in \cite{Gumbel1954Themaxima,Hartley1954Universal} gives us the next theorem.
\begin{theorem}\label{T-LULA}
If in addition to the conditions of Theorem~\ref{T-EXIS}, 
$ \mbox{\rm D}[\tau_{i1}]<\infty $ for each $ i=1,\ldots,n $, then it
holds
\begin{equation}\label{I-LULA}
\mbox{\rm E}[{\cal T}_{1}]
\leq A^{T}\leq
\mbox{\rm E}\Bigg[\bigoplus_{0\leq r+s\leq p}
G^{r}\otimes{\cal T}_{1}\otimes G^{s}\Bigg].
\end{equation}
\end{theorem}

As a consequence, we have the next lemma.
\begin{lemma}\label{L-LUBG}
Under the conditions of Theorem~\ref{T-LULA}, for any finite vector 
$ \bm{x}(0) $, it holds
$$
\|\mbox{\rm E}[{\cal T}_{1}]\|
\leq\gamma\leq
\Bigg\|\mbox{\rm E}\Bigg[\bigoplus_{0\leq r+s\leq p}
G^{r}\otimes{\cal T}_{1}\otimes G^{s}\Bigg]\Bigg\|.
$$
\end{lemma}

\bibliographystyle{utphys}

\bibliography{Products_of_random_matrices_and_queueing_system_performance_evaluation}

\end{document}